\documentclass[reqno,12pt]{amsart}
\usepackage{amscd}

\usepackage{mathrsfs}
\usepackage{indentfirst,latexsym,bm,amsmath,amssymb,amsthm,amscd}
\usepackage{diagrams}
\usepackage[all]{xy}
\newtheorem{theorem}{Theorem}[section]
\newtheorem{lemma}[theorem]{Lemma}

\newtheorem{corollary}[theorem]{Corollary}

\theoremstyle{definition}
\newtheorem{definition}[theorem]{Definition}

\newtheorem{remark}[theorem]{Remark}

\numberwithin{equation}{section}

\begin{document}

\title[canonical degrees of Gorenstein threefolds]
{On the canonical degrees of Gorenstein threefolds of general type}

\author[Rong Du]{Rong Du$^{\dag}$}
\address{Department of Mathematics\\
Shanghai Key Laboratory of PMMP\\
East China Normal University\\
Rm. 312, Math. Bldg, No. 500, Dongchuan Road\\
Shanghai, 200241, P. R. China} \email{rdu@math.ecnu.edu.cn}

\author[Yun Gao]{Yun Gao$^{\dag\dag}$}

\address{Department of Mathematics, Shanghai Jiao Tong University,
Shanghai 200240, P. R. of China}
\email{gaoyunmath@sjtu.edu.cn}

\thanks{$^{\dag}$ The Research Sponsored by the National Natural Science Foundation of China (Grant No. 11471116), Science and Technology Commission of Shanghai Municipality (Grant No. 13dz2260400) and Shanghai Pujiang Program (Grant No. 12PJ1402400).}
\thanks{$^{\dag\dag}$The Research Sponsored by the National Natural Science Foundation of China (Grant No. 11271250,11271251) and SMC program of Shanghai Jiao Tong University.}
\thanks{Both authors are supported by China NSF (Grant No. 11531007)}

 \maketitle

 \begin{abstract}{Let $X$ be a Gorenstein minimal projective $3$-fold with at worst locally factorial terminal singularities. Suppose that the canonical map is generically finite onto its image. C. Hacon showed that the canonical degree is universally bounded by $576$. We improved Hacon's universal bound to $360$. Moreover, we gave all the possible canonical degrees of $X$ if $X$ is an abelian cover over $\mathbb{P}^3$ and constructed all the examples with these canonical degrees.  }\end{abstract}

\section{Introduction}\label{secintro}
The study of the canonical maps of projective varieties of general type is one of the central problems in algebraic geometry. For the case of surfaces, Persson (\cite{Per}) constructed a surface of general type with canonical degree $16$ in 1978. About the same time,
Beauville (\cite{Bea}) proved that the degree of the canonical map is less than or equal to $36$ and with the equality holds if and only if $X$ is a ball quotient surface with $K_X^2=36$,
$p_g=3$, $q=0$, and $|K_X|$ is base point free. Later, Xiao also found some restrictions on surfaces with high
canonical degrees (\cite{Xiao}). Since
the canonical degree is bounded above, next interesting question is to determine
which positive integers $d$'s occur as the degrees of the canonical
map. There are plenty of examples
(see \cite{Bea}, \cite{Cat1}, \cite{V-Z} ) with canonical degrees
being $2$. For $d=3$ and $d=5$, Tan (\cite{Tan}) and Pardini
(\cite{Par1}) constructed several surfaces independently. When
$p_g(\Sigma)=0$, Beauville (\cite{Bea}) constructed surfaces with
$\chi(\mathcal{O}_X)$ arbitrarily large and the canonical degrees
$2,4,6$ and $8$ . For $d=9$, Tan also constructed a surface in
\cite{Tan}. Later, Casnati (\cite{Cas})constructed surfaces of canonical degree from $3$ to $9$ as subvarieties in some projective bundle given by Pfaffans of alternating matrices. The authors (\cite{D-G}) classified the surfaces whose canonical maps are abelian covers over $\mathbb{P}^2$ and constructed these surfaces by explicit defining equations. Recently, Rito in his series papers (\cite{Rit},\cite{Rit2},\cite{Rit3}) constructed some new surfaces of general type with canonical degrees $12$, $16$ and $24$ respectively.

In dimension at least three, the situation seems much less clear. M. Chen studied the canonical map of fiber type (\cite{Ch},\cite{C-H}) and posted an open problem in \cite{Ch} as follows: Let $X$ be a Gorenstein minimal projective $3$-fold with at worst locally factorial terminal singularities. Suppose that the canonical map is generically finite onto its image. Is the generic degree of the canonical map universally upper bounded? Hacon gave a positive answer to Chen's problem. More precisely, he showed that the canonical degree is at most $576$.

In this paper, we improve Hacon's upper bound by showing the following main theorem.
\begin{theorem}\label{maint}
Let $X$ be a Gorenstein minimal complex projective $3$-fold of
general type with locally factorial terminal singularities. Suppose
that $|K_X|$ defines a generically finite map $\phi_X:
X\dashrightarrow \mathbb{P}^{p_g-1}$, then $deg\ \phi_X\le 360$ and with the equality holds if and only if $p_g(X)=4$, $q(X)=2$, $\chi(\omega_X)=5$, $K^3=360$ and $|K_X|$ is base point free.
\end{theorem}

Since  the canonical degree is bounded above, it is quite interesting to consider a parallel problem as surfaces that which positive integers $d$'s occur as the degrees of the canonical map of Gorenstein minimal projective $3$-fold. As far as we know, there are quite few examples about $3$-fold of general type with higher canonical degree. Cai (\cite{Cai}) constructed some examples of $3$-fold with canonical degrees $32$ and $64$ based on the existence of the surface with canonical degree $16$ which was constructed by Persson.
We show that if the canonical map is an abelian cover
over $\mathbb{P}^3$ then the only possible canonical degrees of a
Gorenstein minimal projective $3$-fold are $2^m$ $(1\leqslant m \leqslant 5)$,
by explicit constructions.

%

\section{Proof of the main theorem}
Let $X$ be a Gorenstein minimal complex projective $3$-fold of
general type with locally factorial terminal singularities. Suppose
that $|K_X|$ defines a generically finite map $\phi_X:
X\dashrightarrow \mathbb{P}^{p_g-1}$. We will base on Hacon's beautiful arguments to improve the universal upper bound of the canonical degree.

\begin{proof}
Since $\phi_X$ is generically finite, one has that $p_g(X)\ge 4$. Let $d=deg\ \phi_X$. By the Miyaoka-Yau inequality (\cite{Mi}), we have
\[d(p_g(X)-3)\le K_X^3\le 72\chi(\omega_X).\]
If we can show $\chi(\omega_X)\le p_g(X)+1$, then
\begin{equation}\label{star}
d\le 72\frac{\chi(\omega_X)}{p_g(X)-3}\le 72\frac{p_g(X)+1}{p_g(X)-3}=72(1+\frac{4}{p_g(X)-3})\le360.
\end{equation}
If $q(X)\le2$, then $\chi(\omega_X)\le p_g(X)+q(X)-1\le p_g(X)+1$.

Now we can assume hereafter that $q(X)\ge 3$. Consider the Albanese map $alb_X$ of $X$ and the Stein factorization $f$ of $alb_X$ as follows:

\[\xymatrix{
X \ar[r]^{f}\ar[dr]_{alb_X}
& Y \ar[d]\\
& Alb(X)
}.\]

By Hacon's argument (see the proof of \cite{Ha}, Theorem 1.1), one has
\begin{enumerate}
\item [(1)] $\chi(\omega_X)\le p_g(X)$, if dim$Y\ge 2$;
\item [(2)] $\chi(\omega_X)\le p_g(X)+\chi(\omega_Y)$ and $\chi(\omega_Y)p_g(F)\le p_g(X)$, where $F$ is the general fiber of $f$, if dim$Y=1$.
\end{enumerate}
Hence if dim$Y\ge 2$, by (\ref{star}), the statement holds. More precisely,
\[d\le 72\frac{p_g(X)}{p_g(X)-3}\le 288.\]

We only need to consider dim$Y=1$.

If $p_g(F)\le$ dim$X-1$, then
\[h^0(\mathcal{O}_X(K_X)\otimes\mathcal{O}_F)\le h^0(\mathcal{O}_X(K_X+F)\otimes\mathcal{O}_F)=h^0(\mathcal{O}_F(K_F))\le\text{dim} X-1,\]
which means that dim Im$(\phi_X|_F)\le$ dim $X-2$, and hence dim Im$\phi_X\le$ dim $X-1$, which contradicts the assumption that $\phi_X$ is generically finite. So we have that $p_g(F)\ge$ dim$X=3$ and then $p_g(X)\ge\chi(\omega_Y)p_g(F)=(q(X)-1)p_g(F)\ge 6$.

Therefore
\[d\le 72\frac{\chi(\omega_X)}{p_g(X)-3}\le 72\frac{p_g(X)+\chi(\omega_Y)}{p_g(X)-3}\le 72(1+\frac{1}{p_g(F)})\frac{p_g(X)}{p_g(X)-3}\le192.\]

From the argument above, we know that the equality of (\ref{star}) holds if and only if $p_g(X)=4$, $q(X)=2$, $\chi(\omega_X)=5$, $K_X^3=360$ and $|K_X|$ is base point free.
\end{proof}
\begin{remark}
If $X$ is nonsingular and the canonical divisor $K_X$ is ample then the equality in the main theorem holds if and only if $X$ is a ball quotient. We guess that such a ball quotient with those invariants exists.  For the parallel case of surfaces with the maximal canonical degree, the surface of general type with canonical degree $36$ does exist which was constructed as some fake projective plane by S. Yeung recently (\cite{Yeung}).
\end{remark}

\section{Canonical maps defined by abelain covers}
The theory of cyclic covers of algebraic surfaces was studied first
by Comessatti in \cite{Com}. Then F. Catanese (\cite{Cat2}) studied
smooth abelian covers in the case $(\mathbb{Z}_2)^{\oplus2}$ and R.
Pardini (\cite{Par2}) analyzed the general case. In this section, we
shall recall some basic definitions and results for abelian covers
and construct minimal $3$-folds of general type whose canonical maps are abelian covers over $\mathbb{P}^3$. Since our point of
view is to find the defining equations, we use the second author's
 expressions and notations appearing
in \cite{Gao}.

Let $\varphi:X\to Y$ is an abelian cover associated to abelian group
$G\cong\mathbb Z_{n_1}\oplus\cdots\oplus\mathbb Z_{n_k}$, i.e.,
function field $\mathbb C(X)$ of $X$ is an abelian extension of the
rational function field $\mathbb C(Y)$ with Galois group $G$.  Without lose of
generality, we can assume $n_1|n_2\cdots|n_k$.

\begin{definition}
The dates of abelian cover over $Y$ with group $G$ are $k$ effective
divisors $D_1$, $\cdots$, $D_k$ and $k$ linear equivalent relations
\[D_1\sim n_1L_1, \cdots, D_k\sim n_kL_k.\]
\end{definition}

Let $\mathscr{L}_i=\mathscr{O}_Y(L_i)$ and $f_i$ be the defining
equation of $D_i$, i.e., $D_i=\text{div}(f_i)$, where $f_i\in H^0(Y,
\mathscr{L}_i^{n_i})$. Denote
$\textbf{V}(\mathscr{L}_i)=\textbf{Spec}S(\mathscr{L}_i)$ to be the
line bundle corresponding to $\mathscr{L}_i$, where
$S(\mathscr{L}_i)$ is the sheaf of symmetric $\mathscr{O}_Y$
algebra. Let $z_i$ be the fiber coordinate of
$\textbf{V}(\mathscr{L}_i)$. Then the abelian cover can be realized
by the normalizing of surface $V$ defined by the system of equations
\[z_1^{n_1}=f_1, \cdots, z_k^{n_k}=f_k.\]
So we have the following diagram:

\begin{diagram}
X     & \rTo^{\text{normalization}}   & V & \rdTo^f \rInto  &\oplus_{i=1}^k\textbf{V}(\mathscr{L}_i)\\
      & \rdTo(4,2)_\varphi            &   & \rdTo           &\dTo^p\\
      &                               &   &                 &Y .\\
\end{diagram}
Sometimes we call $X$ is defined by these equations if there is no
confusions in the context.

 We summerize our main results as follows.
\begin{theorem}\label{F}(See \cite{Gao}) Denote by $[Z]$ the integral part of a
$\mathbb Q$-divisor $Z$, $-L_g=-\sum\limits_{i=1}^k g_i{L}_i+
\left[\sum\limits_{i=1}^{k}\frac{g_i}{n_i}D_i\right]$.Then
\begin{eqnarray}
\varphi_*\mathcal{O}_X&=\bigoplus\limits_{{g\in G}}\mathcal
O_Y(-L_g).
\end{eqnarray}
where $g=(g_1,\cdots, g_k)\in G$.

\end{theorem}

 So the
decomposition of $\varphi_*\mathcal{O}_X$ is totally determined by
the abelian cover.

\begin{corollary}\label{h}
If $X$ is non-singular, $D$ is the divisor on $Y$, then
$$h^i(X, \varphi^*\mathcal{O}_Y(D))=\sum\limits_{g\in G}h^i(Y, \mathcal{O}_Y(D-L_g))$$
\end{corollary}


If the canonical map of $X$ is an abelian cover over $\mathbb{P}^3$ then we have the explicit decomposition of $\varphi_*\mathcal{O}_X$.

\begin{lemma}\label{decom} If $\varphi=\phi_{X}$
is a finite abelian cover of degree $d$ over $\mathbb{P}^3$, then
$\varphi_*\mathcal{O}_X=\mathcal{O}_{\mathbb{P}^3}\oplus\mathcal{O}_{\mathbb{P}^3}(-2)^{\oplus
d/2-1}\oplus\mathcal{O}_{\mathbb{P}^3}(-3)^{\oplus
d/2-1} \oplus\mathcal{O}_{\mathbb{P}^3}(-5)$.
\end{lemma}

\begin{proof}
Because $\varphi$ is a finite abelian cover, $\varphi_*\mathcal{O}_X$ is a
direct sum of the line bundles by Theorem \ref{F}.
$$\varphi_*\mathcal{O}_X=\mathcal{O}_{\mathbb{P}^3}\oplus
\bigoplus_{i=1}^{d-1}\mathcal{O}_{\mathbb{P}^3}(-l_i).$$

 Assume
$0<l_{d-1}\leqslant l_{d-2} \leqslant \cdots \leqslant l_1 $.

Since
$K_X=\varphi^*(\mathcal{O}_{\mathbb{P}^3}(1))$, for any $m\geqslant
1$,
$$P_m(X)=h^0(mK_X)=h^0(\varphi^*(\mathcal{O}_{\mathbb{P}^3}(m)))
=h^0(\mathcal{O}_{\mathbb{P}^3}(m))+
\sum_{i=1}^{d-1}h^0(\mathcal{O}_{\mathbb{P}^3}(m-l_i)).$$

Because
$p_g(X)=h^0(\varphi^*(\mathcal{O}_{\mathbb{P}^3}(H)))=h^0(\mathcal{O}_{\mathbb{P}^3}(1))=4$,
we see that
$$h^0(\mathcal{O}_{\mathbb{P}^3}(1-l_i))=0, \quad 1\le i \le d-1.$$

So $l_i\geqslant 2$.

And $p_g=h^3(\varphi_*\mathcal{O}_X)=h^3(\mathcal{O}_{\mathbb{P}^3})
+\sum_{i=1}^{d-1}h^3(\mathcal{O}_{\mathbb{P}^3}(-l_i))$,

So $4=\sum_{i=1}^{d-1}h^0(\mathcal{O}_{\mathbb{P}^3}(l_i-4))$, then
$l_i \leqslant 5$.

Therefore, we have two cases as follows:

\begin{itemize}
\item[(1)]$l_{1}=5,\quad l_2, \cdots, l_{d-1}<4$, and
\item[(2)]$l_{1}=l_{2}=l_{3}=l_{4}=4,\quad l_5, \cdots, l_{d-1}<4$.
\end{itemize}
Let $m=2$, we have the second plurigenus of $X$
$P_2(X)=\chi(2K_X)=\frac{1}{2}K_X^3+\frac{1}{6}K_Xc_2+\chi(\mathcal{O}_X)=\frac{d}{2}+3\chi(K_X)=\frac{d}{2}+9$. On the other hand, $P_2(X)=h^0(\mathcal{O}_{\mathbb{P}^3}(2))+
\sum_{i=1}^{d-1}h^0\big(\mathcal{O}_{\mathbb{P}^3}(2-l_i)\big)$. So
$$
\sum_{i=1}^{d-1}h^0(\mathcal{O}_{\mathbb{P}^3}(2-l_i))=\frac{d}{2}-1,$$
then there are exact $(\frac{d}{2}-1)$ $2$'s among $l_i$'s.

Let $m=3$, we have the third plurigenus of $X$
$P_3(X)=\chi(3K_X)=\frac{5}{2}K_X^3+\frac{1}{4}K_Xc_2+\chi(\mathcal{O}_X)=\frac{5d}{2}+5\chi(K_X)=\frac{5d}{2}+15$. On the other hand, $P_3(X)=h^0(\mathcal{O}_{\mathbb{P}^3}(3))+
\sum_{i=1}^{d-1}h^0\big(\mathcal{O}_{\mathbb{P}^3}(3-l_i)\big)$. So
$$
\sum_{i=1}^{d-1}h^0(\mathcal{O}_{\mathbb{P}^3}(3-l_i))=\frac{5d}{2}-5,$$
The second case does not satisfy the equation. So the lemma is
proved.
\end{proof}

Now let $\varphi: X\rightarrow \mathbb{P}^3$ be an abelian cover
associated to an abelian group $G\cong\mathbb
Z_{n_1}\oplus\cdots\oplus\mathbb Z_{n_k}$. Then $X$ is the
normalization of the $3$-fold defined by
\[z_1^{n_1}=f_1=\prod_\alpha p_\alpha^{\alpha_1}, \cdots, z_k^{n_k}=f_k=\prod_\alpha p_\alpha^{\alpha_k},\]
where $p_{\alpha}$'s are coprime and $\alpha=(\alpha_1,\cdots,
\alpha_k)\in G$. Denote $x_\alpha$ to be the degree of $p_\alpha$,
$e_i=(0,\cdots,0,1,0,\cdots,0)\in G$, $1\leq i \leq k$, and $l_g$ be
the degree of $L_g$, $g\in G$. So $x_\alpha$ and $l_g$ are all integers. Then
\begin{eqnarray}\label{s}
&n_il_{e_i}=\sum\limits_\alpha \alpha_ix_\alpha\quad i=1,\cdots k,\\
& l_g=\sum\limits_{i=1}^k g_i
l_{e_i}-\sum\limits_\alpha\left[\sum\limits_{i=1}^k
\frac{\displaystyle{g_i\alpha_i}}{\displaystyle
 {n_i}}\right]x_\alpha.
 \end{eqnarray}

 \begin{lemma}{\label{jfc}}
Using the notation as above, if $\varphi=\phi_{X}$, then there
exists $g'=(g'_1,\cdots, g'_k)\in G\cong\mathbb
Z_{n_1}\oplus\cdots\oplus\mathbb Z_{n_k}$ and a partition of $G$ set-theoradically, $G=\{0\}\cup\{g'\}\cup S_1\cup S_2$, where the cardinalities of $S_1$ and $S_2$ are equal, such that $x_{\alpha}$
satisfies the following equation
$${\LARGE{(*)}}\quad\quad\left\{\begin{array}{l}
n_il_{e_i}=\sum\limits_\alpha \alpha_ix_\alpha\\
l_{g'}=\sum\limits_{i=1}^k g'_i
l_{e_i}-\sum\limits_\alpha\left[\sum\limits_{i=1}^k
\frac{\displaystyle{g'_i\alpha_i}}{\displaystyle
 {n_i}}\right]x_\alpha=5 \hspace{1cm} \\
 l_g=\sum\limits_{i=1}^k g_i
l_{e_i}-\sum\limits_\alpha\left[\sum\limits_{i=1}^k
\frac{\displaystyle{g_i\alpha_i}}{\displaystyle
 {n_i}}\right]x_\alpha=3, \quad \quad g\in S_1\\
l_g=\sum\limits_{i=1}^k g_i
l_{e_i}-\sum\limits_\alpha\left[\sum\limits_{i=1}^k
\frac{\displaystyle{g_i\alpha_i}}{\displaystyle
 {n_i}}\right]x_\alpha=2, \quad \quad g\in S_2
\end{array}\right.$$
\end{lemma}
\begin{proof} It comes from Lemma \ref{decom} directly.

\end{proof}
By the above lemma, finding $3$-folds whose canonical map are abelian
covers over $\mathbb{P}^3$ is equivalent to finding the integral
roots $\{x_\alpha\}$ of the above equations.

\begin{theorem}
Let $X$ be a Gorenstein minimal complex projective 3-fold of general type with locally factorial terminal singularities and $\varphi: X\longrightarrow\mathbb{P}^3$ is an abelian cover.
If $\varphi=\phi_{X}$ then the canonical degree can only be $2^m$, $1\leqslant m \leqslant 5$.
\end{theorem}

\begin{proof}
By Lemma \ref{jfc}, we only need to find the integral
solutions of the equations (*). So the only possible degrees are $2, 4, 8, 16$ and $32$ by using computer calculations.
Moreover, we have the defining equations of the examples of all the degrees as follows.

Degree $2$: $$z^2=f;$$

Degree $4$:
$$\left\{\begin{array}{l}
z_1^2=s\\
z_2^2=q;
\end{array}\right.$$

Degree $8$:

$$\left\{\begin{array}{l}
z_1^2=t_1q\\
z_2^2=t_2q\\
z_3^2=t_3q;
\end{array}\right.$$

Degree $16$:
$$\left\{\begin{array}{l}
z_1^2=h_1h_4t_1t_2\\
z_2^2=h_2h_4t_2t_3\\
z_3^2=h_3h_4t_1t_3\\
z_4^2=h_2h_3t_3;
\end{array}\right.$$


Degree $32$:
$$\left\{\begin{array}{l}
z_1^2=h_{1}h_{2}h_{3}h_{10}\\
z_2^2=h_{4}h_{5}h_{6}h_{10}\\
z_3^2=h_{2}h_{3}h_{6}h_{7}\\
z_4^2=h_{1}h_{3}h_{5}h_{8}\\
z_5^2=h_{7}h_{8}h_{9}h_{10};
\end{array}\right.$$
where the degree of $h$'s is $1$, $t$'s is $2$, $q$'s is $4$, $s$'s is $6$, $f$'s is $10$ and they all define nonsingular surfaces in $\mathbb{P}^3$ and intersect normal crossingly.

We want to show these 3-folds are smooth after normalization.
Since the arguments are similar, we only prove the most complicated case with canonical degree $32$.

Actually, we only need to consider the intersections of the branch locus locally.
Let $\ell_{ij}$'s be the intersection lines of $h_i$ and $h_j$. Around the general point of $\ell_{ij}$ (except the intersection of three planes), the
cover is locally defined by \[z_1^2=x^{a_{11}}y^{a_{12}}, \quad
z_2^2=x^{a_{21}}y^{a_{22}}, \quad z_3^2=x^{a_{31}}y^{a_{32}},\quad
z_4^2=x^{a_{41}}y^{a_{42}},\] where $a_{ij}=0$ or $1$ for all $i$,
$j$.

It is easy to check that $\{(a_{i1}, a_{i2})\}\nsubseteqq
\{(1,1),(0,0)\}$ i.e., at least one pair $\{(a_{i1},
a_{i2})\}=\{(1,0)\}$ or $\{(0, 1)\}.$  Without lose of generality,
we can assume $(a_{11},a_{21})=(1,0)$, i.e. $z_1^2=x$. After normalization, the cover is
branched along the smooth surfaces. So the 3-fold is smooth at the preimages of the general points of $\ell_{ij}$'s under the normalization map.

Let $p_{ijk}$ be the intersection point of $h_i$, $h_j$ and $h_k$. The arguments are similar. Let us take $p_{123}$ for example.  The
cover is locally defined by

\[z_1^2=xyw, \quad z_3^2=yw,\quad
z_4^2=w.\]

After normalization, the cover
is locally defined by

\[z_1^2=x, \quad z_3^2=y,\quad
z_4^2=w.\]

So the 3-folds are smooth.

It is easy to see that these $3$-folds are all nonsingular
with $p_g(X)=4$, $q(X)=h^{2,0}=0$, $\chi(\mathscr{O}_X)=-3$, $K_X=\varphi^*(\mathscr{O}_{\mathbb{P}^3}(1))$ and $K_X^3$ equals each degree of the covers.
\end{proof}

\begin{remark}
By Lemma \ref{jfc}, we have the integral
solution of the equations (*) for degree $6$ and $18$. But the isolated singularities of corresponding $3$-fold are not Gorenstein terminal since they are not cDV (\cite{Rei2}).
\end{remark}

\section*{Acknowledgements}
Both authors would like to thank N. Mok for
supporting their researches when they were in the University of Hong Kong and the referees for the useful comments.

\end{document}